\documentclass[reqno,centertags,10pt,standalone]{amsart}
\usepackage[letterpaper,margin=1.3in]{geometry}
\usepackage{amsmath,amsthm,amsfonts,amssymb,enumerate}
\usepackage[bookmarksopen=true,final]{hyperref}
\usepackage{multirow,yhmath}
\usepackage{xcolor}
\usepackage{bm}
\usepackage[pdftex]{graphicx}
\numberwithin{figure}{section}
\newcommand{\bbC}{{\mathbb{C}}}

\newcommand{\bbR}{{\mathbb{R}}}

\newcommand{\bbZ}{{\mathbb{Z}}}

\newcommand{\calJ}{{\mathcal J}}

\newcommand{\Arg}{\text{\rm{Arg}}}

\newcommand{\rank}{\text{\rm{rank}}}
\newcommand{\ran}{\text{\rm{Ran}}}
\newcommand{\sgn}{\text{\rm{sgn}}}

\newcommand{\beq}{\begin{equation}}
\newcommand{\eeq}{\end{equation}}
\newcommand{\ba}{\begin{align*}}
\newcommand{\ea}{\end{align*}}

\DeclareMathOperator{\re}{Re}
\DeclareMathOperator{\im}{Im}

\allowdisplaybreaks

\newcommand{\R}{\mathbb R}

\numberwithin{equation}{section}
\newtheorem{theorem}{Theorem}
\newtheorem{lemma}[theorem]{Lemma}

\newtheorem{corollary}[theorem]{Corollary}
\newtheorem{proposition}[theorem]{Proposition}

\theoremstyle{definition}

\theoremstyle{remark}
\newtheorem{remark}[theorem]{Remark}

\numberwithin{theorem}{section}
\begin{document}

\title[A generalized Hermite-Biehler theorem]{A generalized Hermite--Biehler theorem}

\author[R.~Kozhan]{Rostyslav Kozhan}
			
\address
{Department of Mathematics, Uppsala University, Uppsala, Sweden}
		
\email{rostyslav.kozhan@math.uu.se}

%----------Author 3
\author[M.~Tyaglov]{Mikhail Tyaglov}
			
\address
{School of Mathematical Sciences and MOE-LSC, Shanghai Jiao Tong University, Shanghai, P.R.China}
\email{tyaglov@sjtu.edu.cn}%{tyaglov@sjtu.edu.cn}

\date{\today}
		
			
			
\begin{abstract}
The classical Hermite--Biehler theorem describes possible zero sets of complex linear combinations of two real polynomials whose zeros strictly interlace. We provide the full characterization of zero sets for the case when this interlacing is broken at exactly one location.

Using this we solve the direct and inverse spectral problem for rank-one multiplicative perturbations of finite Hermitian matrices. We also treat certain rank two additive perturbations of finite Jacobi matrices.
\end{abstract}
\maketitle

%%%%%%%%%%%%%%%%%%%%%%%%%%%%%%%%%%%%
\section{Introduction}
%%%%%%%%%%%%%%%%%%%%%%%%%%%%%%%%%%%%

Consider a complex linear combination
\begin{equation}\label{eq:genLin}
\alpha p(z) + \beta q(z), \quad \alpha,\beta \in\bbC\setminus \bbR
\end{equation}
of two real polynomials $p$ and $q$.
We say that zeros of two polynomials strictly interlace if all the zeros are real, simple, and between two consecutive zeros of one polynomial there lies a unique zero of the other polynomial.
The classical Hermite--Biehler theorem, which goes back to 1879~\cite{Biehler,Hermite}, says that zeros of $p$ and $q$ strictly interlace if and only if all zeros of~\eqref{eq:genLin} simultaneously belong to the upper or lower complex half-plane. This result has numerous applications both in pure and applied mathematics, most notably in control theory.

We want to classify possible zero configurations of~\eqref{eq:genLin} for the case when zeros of $p$ and $q$ ``almost'' interlace, that is, if we add one point to the zero set of $q$ then it will interlace the zero set of $p$.  The precise statement is in Theorems~\ref{thm:HBextendedNEW} and~\ref{rem:around0}.

Motivation comes from the spectral problem for rank-one {\it multiplicative non-Hermitian} perturbations of Hermitian and/or Jacobi matrices (see Section~\ref{ss:multipl}). Such perturbations were recently considered in random matrix theory~\cite{ORW,AlpK2}.
Certain rank-two additive non-Hermitian perturbations also lead to the polynomials of this type, see Section~\ref{ss:ranktwo}.

%%%%%%%%%%%%%%%%%%%%%%%%%%%%%%%%%%%%%%%%%%%%%%%%%%%%%%%%%%%%%%%%%
\section{Hermite-Bieheler theorem}
%%%%%%%%%%%%%%%%%%%%%%%%%%%%%%%%%%%%%%%%%%%%%%%%%%%%%%%%%%%%%%%%%

Let
\begin{equation}\label{main.poly}
h(z)=\prod\limits_{j=1}^{n}(z-z_j)
\end{equation}
be a monic complex polynomial. In the complex plane~$\mathbb{C}$, consider
the curve $\Gamma_h\equiv\left\{\,h(t)\,\colon\,t\in\mathbb{R}\,\right\}$
and specify the direction on it, which corresponds to increasing of
parameter~$t$. This oriented curve is called (Mikhailov--Nyquist) {\it hodograph} of the
polynomial $h$, see. e.g.,~\cite{Barkovsky,Postnikov}. Let us assume that polynomial $h$ has
no roots on the real axis. In that case, $\Gamma_h$ does not go through the origin and the function
\begin{equation}\label{eq:a}
\varphi_h(t):= \sum_{j=1}^n \Arg (t-z_j), \quad t\in\bbR,
\end{equation}
where $\Arg\,z\in[-\pi,\pi)$ is the principal value of the argument of $z\in\bbC\setminus\{0\}$,
is {\it continuous} at each point of the real axis. It is clear that $\varphi_h(t)$
is a branch of $\arg h(t)$, $t\in\bbR$. % and defined uniquely, see~\cite{Postnikov}.
Let us introduce
the increment
\begin{equation}\label{eq6}
\Delta_h:=\bigl.\varphi_h\bigr|_{-\infty}^{+\infty}.
\end{equation}
The existence of the limits will be shown below.

%Let us formulate a simple lemma that will be used below.
%
%\begin{lemma}\label{lem:0}
%	Let $p(z)$ and $q(z)$ be two real polynomials %with simple zeros
%	and $\alpha,\beta\in\bbC$ satisfy $\im(\bar\alpha\beta)\ne0$.
%	Suppose that all zeros %$\{z_j\}_{j=1}^n$ of
%	$h(z) = \alpha p(z) + \beta q(z)$  are  in $\bbC\setminus\bbR$. Then $\varphi_h(t)$, see~\eqref{eq:a}, satisfies $\varphi_h'(t)\ne0$ for any $t$ that is a simple real zero of $p(z)$ or $q(z)$.
%\end{lemma}
%\begin{proof}
%	 We can renormalize $h$ to make it monic (notice that $\im(\bar\alpha\beta)\ne0$ is preserved), so that we can assume that~\eqref{main.poly} holds. $\varphi_h(t)$  is $\tan^{-1} \frac{\im  h(t)}{\re h(t)} + 2\pi m$ for some $m\in\bbZ$, which implies
	 %$$
	 %a'(t) = \frac{(\im h(t))' \re h(t) - (\re h(t))' \im h(t)}{(\im h(t))^2 + (\re h(t))^2}.
	 %$$
	 %Using  $\im h(t) = \im\alpha p(t) + \im\beta q(t)$ and $\re h(t) = \re\alpha p(t) + \re\beta q(t)$ and some manipulations, we get
%	 \begin{equation}\label{eq:derivativeVarphi}
%	 	\varphi_h'(t)
	 	%=(\im\alpha \re\beta-\im\beta\re\alpha) \frac{p'(t) q(t) -p(t) q'(t)}{|h(t)|^2}
%	 	=\frac{(\im h(t))' \re h(t) - (\re h(t))' \im h(t)}{(\im h(t))^2 + (\re h(t))^2}
%	 	=(\im(\bar\alpha\beta)) \frac{p(t) q'(t) -p'(t) q(t)}{|h(t)|^2} .
%	 \end{equation}
%	 Now, if $p(t) =0$ with $t\in\bbR$ then $q(t)\ne 0$ since $h(t)\ne 0$. Then  $\varphi_h'(t)=0$ if and only if $p'(t)=0$ which is impossible if $t$ is a simple zero of $p$.
%	  Similarly, $q(t)=0$ with $t\in\bbR$ implies that $\varphi_h'(t)=0$ if and only if $q'(t)=0$.
%\end{proof}

The following theorem is usually attributed to
C.\,Hermite. Let us denote the open upper half-plane by $\bbC_+ = \{z\in\bbC: \im z>0\}$ and the lower open half-plane by $\bbC_- = \{z\in\bbC: \im z<0\}$.
\begin{theorem}[Hermite]\label{thm2.1}
If the polynomial $h$ has $n_{+}$ roots in $\bbC_+$, $n_{-}$ roots
in $\bbC_-$, and no roots on the real axis, then
\begin{equation*}%\label{eq7}
\Delta_h=\pi(n_{+}-n_{-}).
\end{equation*}
\end{theorem}
\begin{proof}
Let $h_j(z):=z-z_j$. Then $\varphi_{h_j}(t)=\Arg (t-z_j)$. It is easy to see that the
hodograph~$\Gamma_{h_j}$ of~$h_j$ is a horizontal line traversed from left to right that
intersects the imaginary axis at the point $-i\im z_j$. Thus, as $t$ runs from $-\infty$
to $\infty$, the radius-vector of a point on the hodograph makes a counterclockwise turn
of magnitude $\pi$ if $\im z_j>0$ (clockwise if $\im z_j<0$), that is,
$\Delta_{h_j}=\pi\,\sgn\im z_j$. Now from~\eqref{eq:a}--\eqref{eq6} we obtain
\begin{equation*}
\Delta_h=\sum\limits_{j=1}^n\Delta_{h_j}=\pi\sum\limits_{j=1}^n\sgn\im\lambda_j=\pi(n_{+}-n_{-}).
\end{equation*}
\end{proof}
\noindent From the proof of the theorem we obtain that the number $\Delta_h$ is well defined, and
\begin{equation}\label{limits}
\lim\limits_{t\to-\infty}\varphi_h(t)
=-\Delta_h
=\pi(n_{-}-n_{+}),\quad \lim\limits_{t\to+\infty}\varphi_h(t)=0.
\end{equation}

\vspace{3mm}

Let us consider the extremal case when all roots of the polynomial $h$ lie in $\bbC_+$.
In this case, the polynomial $h$ has remarkable properties. In particular, from Theorem~\ref{thm2.1}
we immediately have the following consequences.
\begin{corollary}\label{corol.HB}
Let $h$ be as in~\eqref{main.poly} with $z_j\notin\bbR$ for all $j$. Then $h$ has all roots in $\bbC_+$ if and only if $\Delta_h=\pi n$.
\end{corollary}

\begin{lemma}[Hermite]\label{Lemma.HB.poly}
If all the zeros of $h$ lie in $\bbC_+$, then $\varphi_h(t)$
is a monotone (strictly) increasing function on $\mathbb{R}$.
\end{lemma}

\begin{proof}
Indeed, for $h_j(z)=z-z_j$, with $\im z_j>0$, one has
\begin{equation*}
\dfrac{d\varphi_{h_j}(t)}{dt}=\dfrac{\im z_j}{(t-\re z_j)^2+(\im z_j)^2}>0.
\end{equation*}
Thus, $\varphi_{h_j}(t)$ is strictly increasing function on $\mathbb{R}$. Then so is
the function $\varphi_h=\sum \varphi_{h_j}$, see~\eqref{eq:a}.
\end{proof}

As was independently established by C.\,Hermite~\cite{Hermite} and M.\,Biehler~\cite{Biehler} in 1879,
such a property of the polynomial $h$ can be restated in terms of the root locations of the real and imaginary parts
of $h$. Namely, after normalization we can express the polynomial $h$~\eqref{main.poly} as
\begin{equation}\label{eq:nonHermPoly}
h(z)=p(z) - il q(z),
\end{equation}
where $p$ and $q$ are monic real polynomials with $\deg p=n$, $\deg q\leqslant n-1$, and $l\in\bbR$.
%We say that zeros of two polynomials strictly interlace if between two consecutive zeros of one polynomial lies a unique zero of the other polynomial.

Then the following theorem holds whose brief proof we provide for completeness purposes.

\begin{theorem}[Hermite--Biehler theorem \cite{Hermite,Biehler}]\label{thm:HB}
Let a polynomial $h$ be given by~\eqref{eq:nonHermPoly} with $p$ and $q$ monic and real, $\deg p=n$, $\deg q\leqslant n-1$, $l\in\bbR$.
Then all zeros of $h$ belong to  $\bbC_+$ if and only if $l>0$ and zeros of
$p$ and $q$
%are real, simple, and
strictly interlace.
\end{theorem}
\begin{proof}
Let all the zeros of $h$ lie in $\bbC_+$. Then by~\eqref{limits}
and  Lemma~\ref{Lemma.HB.poly}, $\varphi_h(t)$ is a strictly monotone increasing function
with $\lim\limits_{t\to-\infty}\varphi_h(t)=-\pi n$ and $\lim\limits_{t\to+\infty}\varphi_h(t)=0$. Therefore, there are
exactly $n$ real values $\{\lambda_j\}_{j=1}^n$ with $\varphi_h(\lambda_j)=\tfrac\pi2 +k\pi$
and $n-1$ real values $\{\mu_j\}_{j=1}^{n-1}$ with $\varphi_h(\mu_j)=k\pi$, and these two
sets strictly interlace.  From~\eqref{eq:nonHermPoly} it follows that $p(\lambda_j)=0$ and
$q(\mu_j)=0$ for each $j$. Finally, since $l = \sum\limits_{j=1}^n \im z_j$ we get $l>0$.

%Conversely, suppose that the zeros of $p$ and $q$ are real, simple, and interlace, and $l>0$.
%Then the function

Conversely, suppose $p(\lambda_j)=0$ and $q(\mu_j)=0$, where $\{\lambda_j\}_{j=1}^n$ and
$\{\mu_j\}_{j=1}^{n-1}$ strictly interlace, so that $\lambda_j<\mu_j<\lambda_{j+1}$.
Notice that $z_j\notin\bbR$ for any $j$, otherwise, $p(z_j) = q(z_j)=0$ contrary to strict
interlacing, so that $\varphi_h(t)$ is well defined for all $t\in\bbR$.
It is not hard to see that $p'(t)q(t)-p(t)q'(t)>0$ (see, e.g.,~\cite[Thm~3.4]{Holtz_Tyaglov}).
Since $\varphi_h(t)$ is a branch of $\arg h(t)$ we get %see~\eqref{eq:derivativeVarphi},
$$
\varphi_h'(t) = \frac{d}{dt} \left(\arctan \frac{-lq(t)}{p(t)} \right)  = l\frac{p'(t)q(t)-p(t)q'(t)}{|h(t)|^2}>0
$$
for all $t\in\bbR$. % (here $\im(\bar\alpha\beta) = -l<0$).
%As we mentioned above,
%$\varphi_h(t)$ is a branch of $\arg h(t)$, therefore,
%
%\begin{equation*}
%$\varphi_h(t)=\tan^{-1}\left(\frac{-lq(t)}{p(t)}\right)+2\pi m$
%\end{equation*}
%
%for some $m\in\mathbb{N}$, and
%
%\begin{equation*}
%\varphi'_h(t)=l\cdot\dfrac{p'(t)q(t)-p(t)q'(t)}{p^2(t)+l^2q^2(t)}.
%\end{equation*}
%
%By~\cite[Theorem~3.4]{Holtz_Tyaglov}, we have $p'(t)q(t)-p(t)q'(t)>0$, and since $l>0$,
%it follows that $\varphi'_h(t)>0$ for any $t\in\mathbb{R}$.
Thus, $\varphi_h(t)$
is strictly increasing on $\mathbb{R}$, and intersects the imaginary axis
exactly~$n$ times at the points $\lambda_j$, $j=1,\ldots,n$, and the real axis
$n-1$ times at the points $\mu_j$, $j=1,\ldots,n-1$. This means that $\Delta_h=\pi n$,
so that all zeros of $h(z)$ lie in $\bbC_+$ according to Corollary~\ref{corol.HB}.
\end{proof}

\begin{remark}\label{rem:1-to-1}
Theorem~\ref{thm:HB} is just one step away from the full classification of complex spectrum of rank-one non-Hermitian perturbations of Jacobi matrices (or any Hermitian matrices with a cyclic vector). We elaborate in Section~\ref{ss:additive} below.
%
%Given a monic complex (not purely real) polynomial $h$, it is clear that there is a unique pair of monic real polynomials $p$ and $q$, and a unique number $l$ such that~\eqref{eq:nonHermPoly} holds.
%Together with Theorem~\ref{thm:HB} this implies that there is a one-to-one correspondence
%between a collections of $n$ points $\{z_j\}_{j=1}^n$ in $\bbC_+$ and a number $l>0$ with
%two collections of interlacing distinct real numbers $\{\lambda_j\}_{j=1}^n$ and $\{\mu_j\}_{k=1}^{n-1}$.
%This fact is important for the solution of an inverse spectral problem for rank-one perturbation
%of Jacobi and Hermitian matrices, see Section~\ref{ss:additive} for details.
\end{remark}

\begin{remark}\label{Remark.HB.2}
The same proof shows that Theorem~\ref{thm:HB} holds if we replace~\eqref{eq:nonHermPoly} with $h(z)=p(z)+\beta q(z)$,
where $\beta$ with $\im\beta<0$ is fixed.
One can also prove an analogue of Theorem~\ref{thm:HB} in the case when
$\deg q=n$. We leave this as an exercise to the reader. %To prove this one can argue along similar lines as our proof of Theorem~\ref{thm:HBextendedNEW}  when $s=0$ or $s=n$.
\end{remark}

If one assumes that $p$ has $s$ negative and $n-s$ non-negative zeros  then {we} can say more on all of the possible locations of zeros of $h$ in~\eqref{eq:nonHermPoly}. This was discussed in~\cite[Sect 3]{K17} but we provide a more general statement and a more transparent proof here. This result is also related to the condition that enters in the case of broken interlacing, see~\eqref{eq:zerosNEW} and Remark~\ref{rem:allpositive} below.
	
\begin{proposition}\label{rem:HBlocalization}
Let the polynomial $h$ defined in~\eqref{eq:nonHermPoly} have all zeros $\{z_j\}_{j=1}^n$ in $\bbC_+$.
\begin{itemize}
\item[(i)] $p$ has $s$ negative zeros and $n-s$ positive ones if and only
if the zeros of $h$ satisfy the inequalities
\begin{equation}\label{eq:HBlocalization}
\frac\pi2+\pi(s-1) < \sum_{j=1}^n \Arg\, z_j < \frac\pi2 +\pi s.
\end{equation}
\item[(ii)] $p$ has $s$ negative zeros, $n-s-1$ positive zeros, and a zero at $t=0$
if and only if the zeros of $h$ satisfy the identity
\begin{equation}\label{eq:HBlocalization2}
\sum_{j=1}^n \Arg\, z_j = \frac\pi2 +\pi s.
\end{equation}
\end{itemize}
\end{proposition}
\begin{proof}
If $\im z_j>0$, $j=1,\ldots,n$, then by Lemma~\ref{Lemma.HB.poly} and Theorem~\ref{thm:HB},
the zeros $\{\lambda_j\}_{j=1}^{n}$ of~$p$ are real and simple, and the function $\varphi_h(t)$
defined in~\eqref{eq:a} is strictly increasing on $\mathbb{R}$ with
$\lim\limits_{t\to-\infty}\varphi_h(t)=-\pi n$ and $\lim\limits_{t\to+\infty}\varphi_h(t)=0$.
Consequently, we have $\varphi_h(\lambda_j)=-\tfrac\pi2-(n-j)\pi$. It is clear now that
\begin{equation*}
-\tfrac\pi2-(n-s)\pi < \varphi_h(0)< -\tfrac\pi2-(n-s-1)\pi
\end{equation*}
if and only if $p$ has $s$ negative zeros and $n-s$ positive zeros. On the other hand,
\begin{equation}\label{eq:a(0)basic}
\varphi_h(0) = \sum_{j=1}^n \Arg(-z_j)
=  \sum_{j=1}^n   \Arg\, z_j  -\pi n.
\end{equation}
The result in (i) follows by combining these two statements.
	
For (ii) observe that $\lambda_{s+1} = 0$, which {holds if and only if}
$\varphi_h(0)=-\tfrac\pi2-(n-s-1)\pi$. Then we use~\eqref{eq:a(0)basic} to get~\eqref{eq:HBlocalization2}.
\end{proof}

%%%%%%%%%%%%%%%%%%%%%%%%%%%%%%%%%%%%%%%%%%%%%%%%%%%%%%%%%%%%%%%%%
\section{Generalized Hermite--Biehler theorem: broken interlacing at $t=0$}
%%%%%%%%%%%%%%%%%%%%%%%%%%%%%%%%%%%%%%%%%%%%%%%%%%%%%%%%%%%%%%%%%

As mentioned (see Section~\ref{ss:additive} for details), the Hermite--Biehler theorem enters naturally in the solution to the
inverse spectral problem for a non-Hermitian rank-one additive perturbation of
Jacobi and Hermitian matrices. %However,  {\it multiplicative} rank-one perturbations or additive rank-two perturbations require analogues the Hermite-Biehler theorem where we no longer have perfect interlacing property of the zeros.
However,  higher rank perturbations require analogues of the Hermite--Biehler
theorem where we no longer have  interlacing  of  zeros of the real and imaginary parts of a given polynomial.

In particular,
{\it multiplicative} rank-one perturbations
and additive rank-two perturbations (see Sections~\ref{ss:multipl} and~\ref{ss:ranktwo}) lead to  situations where  interlacing
of  zeros
%of the real and imaginary parts of a given polynomial
is broken at
exactly one location. We therefore need to consider a complex linear combination of two such polynomials
and study its possible zero configurations.

%Recall that $\Arg \,z\in[-\pi,\pi)$ stands for the principal value of the argument of $z$.
The quantity that will enter naturally is the complex argument of zeros modulo $\pi$. Let us introduce the notation
\begin{equation*}\label{arg.mod.pi}
\Arg_{[0,\pi)} z :=
\begin{cases}
\Arg\,z\qquad&\text{if}\quad z\in\bbC_+,\\
0\qquad&\text{if}\quad z\in\R,\\
\pi+\Arg\,z\,\quad&\text{if}\quad z\in\bbC_-.\\
\end{cases}
\end{equation*}

Geometrically $\Arg_{[0,\pi)} z \in [0,\pi)$ measures the angle between the radius-vector of $z$  and the positive  half-axis $\bbR_+$ if $z\in\bbC_+$ and  the negative half-axis
 $\bbR_-$ if $z\in\bbC_-$. Algebraically one can see that
\begin{equation}\label{eq:cotan}
\Arg_{[0,\pi)} z
= \operatorname{arccot} \frac{\re z}{\im z} \in (0,\pi)
\end{equation}
for $z\in\bbC\setminus\bbR$.

The following theorem is the main result of the paper. It characterizes possible zero configurations of  linear combinations of two monic real polynomials
whose positive and negative zeros interlace except for the broken interlacing at the origin.

\begin{theorem}\label{thm:HBextendedNEW}
Let $p$ and $q$ be monic real polynomials, $\deg p= n$, $\deg q=n-1$, $p(0)\ne 0$, and $\alpha\in\bbC_+$.
Define  the
monic polynomial
\begin{equation}\label{eq:extHBpolyNEW}
h(z) = \alpha p(z) + (1-\alpha) z q(z) = \prod_{j=1}^n (z-z_j).
\end{equation}
If the zeros of $p$ and $q$ %are real, simple, and
strictly interlace,  then  $h$ 	has all of its zeros $\{z_j\}_{j=1}^n$ in
\begin{equation}\label{eq:zerosNEW}
\left\{ \{z_j\}_{j=1}^n \in (\bbC\setminus\bbR)^n: \sum_{j=1}^n \Arg_{[0,\pi)} z_j =\Arg\,\alpha \right\}.
\end{equation}

Conversely, if  $\alpha\in\bbC_+$ and all the zeros of a monic complex polynomial $h(z) = \prod_{j=1}^n (z-z_j)$
satisfy~\eqref{eq:zerosNEW}, then there exists a unique pair of monic real polynomials $p$ and $q$ with strictly interlacing zeroes such that $h(z) = \alpha p(z) + (1-\alpha) z q(z) $.

The number of $z_j$'s in $\bbC_+$ (respectively, in $\bbC_-$) coincides with the
number of positive (respectively, negative) zeros of $p$.
%
%For each fixed {$\alpha\in\mathbb{C}_{+}$} there is one-to-one correspondence between $\{z_j\}_{j=1}^n$ in~\eqref{eq:zerosNEW} and $p$ and $q$ satisfying above.
\end{theorem}
\begin{remark}\label{rem:allpositive}
If all zeros of $p$ are positive or all negative, then we actually have strict interlacing
of zeros of $p(z)$ and $zq(z)$. After dividing by $\alpha$ we end up in the setting of Remark~\ref{Remark.HB.2}.
The extra restriction~\eqref{eq:zerosNEW} on possible zero configuration comes from the fact that $zq(z)$ vanishes at the origin, compare with Proposition~\ref{rem:HBlocalization}(ii), $s=0$.

%, and this case is covered by a modification of the
%Hermite-Biehler theorem, since %$h(z)=[\re\alpha p(z)+(1-\re\alpha)zq(z)]+i\im\alpha(p(z)-zq(z))$, and the zeros
%of the real and imaginary parts of %$h$ are  interlacing by the so-called Obreschkoff or Hermite-Kakeya theorem,
%see, e.g.,~\cite[Theorem 6.3.8]{Rahman.Sch}. {\color{red}\textsc{Shall we cite the statement of the theorem and give more explanations
%here???}}
%However, from the proof of Theorem~\ref{thm:HBextendedNEW} one can see that for a given $\alpha$,
%$\im\alpha>0$, the polynomial $h$ can be uniquely represented as in Theorem~\ref{eq:extHBpolyNEW},
%so this theorem generalizes the Hermite-Biehler theorem, since it provides an additional
%property~\eqref{eq:zerosNEW} for zeros of the polynomial~$h$.
\end{remark}
\begin{proof}

$[\Rightarrow]$
Suppose $p(\lambda_j)=0$ and $q(\mu_j)=0$, where $\{\lambda_j\}_{j=1}^n$
and $\{\mu_j\}_{j=1}^{n-1}$ strictly interlace, so that
$\lambda_j<\mu_j<\lambda_{j+1}$. Let $s$ be the number of negative zeros of $p$, i.e., $\lambda_s<0$  and $\lambda_{s+1}>0$.
Let $h$ be as in~\eqref{eq:extHBpolyNEW} and $z_j$'s be its zeros. Define the corresponding function $\varphi_h$ as in~\eqref{eq:a}.
Since $p(0)\neq0$, $\im \alpha\ne 0$, and  zeros of $p$ and $q$ do not overlap, the polynomial $h$ has no
real zeros, $z_{j}\in\bbC\setminus\bbR$, so that $\varphi_h(t)$ %does not
%traverse through the origin, so
is well defined for all $t\in\bbR$.

Recall that $\im \alpha>0$ {by assumption}. Let us introduce the numbers
\begin{equation}\label{numbers.A}
A_1:=\Arg_{[0,\pi)} (1-\alpha) \in(0,\pi)\qquad
\text{and}\qquad A_2:=\Arg_{[0,\pi)} \alpha =\Arg\,\alpha \in(0,\pi).
\end{equation}
By~\eqref{eq:cotan} $A_1=\operatorname{arccot}\frac{\re \alpha-1}{\im \alpha}$
and $A_2 = \operatorname{arccot}\frac{\re \alpha}{\im \alpha}$. Since
$\im\alpha>0$ and $\operatorname{arccot}$ is monotone decreasing,
we get $\pi>A_1>A_2>0$.

Now rewrite ~\eqref{eq:extHBpolyNEW} for $t\in\bbR$ as
$$
\begin{pmatrix}
        \re h(t) \\
        \im h(t)
\end{pmatrix} =
p(t)
\begin{pmatrix}
        \re\alpha  \\
        \im\alpha
\end{pmatrix}
-tq(t)
\begin{pmatrix}
        \re\alpha-1 \\
        \im\alpha
\end{pmatrix}.
$$
Since $p(t)$ and $tq(t)$ do not vanish simultaneously, this is a non-trivial %($p(t)$ and $tq(t)$ cannot be zero simultaneously)
linear combination of two vectors $\mathbf{v}:=\begin{pmatrix}
        \re\alpha  \\
        \im\alpha
\end{pmatrix}$ and $\mathbf{u}:=\begin{pmatrix}
        \re\alpha-1 \\
        \im\alpha
\end{pmatrix}$. It is parallel to $\mathbf{v}$ if and only if $tq(t)=0$ and to $\mathbf{u}$ if and only if $p(t)=0$. This implies that $\Arg_{[0,\pi)} h(t) =A_1$ if and only if $p(t)=0$ and
$\Arg_{[0,\pi)} h(t) =A_2$ if and only if $tq(t)=0$.

%From~\eqref{eq:extHBpolyNEW} it follows that
%$\Arg_{[0,\pi)} h(t) =A_1$ if and only if $p(t)=0$. Similarly,
%$\Arg\,h(t) \operatorname{mod} \pi=A_2$ if and only if $tq(t)=0$.
Since $\Arg_{[0,\pi)} h(t) =\varphi_h(t)\operatorname{mod} \pi$,
we finally arrive to the central observation that for $t\in\bbR$
%\begin{align}
%\label{eq:intersection1}    \varphi_h(\lambda_j) &=A_1+\pi m_j \quad \mbox{ for some } m_j\in\bbZ, \\
%\label{eq:intersection2}
%    \varphi_h(\mu_k)     &=A_2+\pi r_k\quad \mbox{ for some } r_k\in\bbZ,
%\end{align}
\begin{alignat}{3}
\label{eq:intersection1}
%&\varphi_h(t) =A_1+\pi m \mbox{ for some }  m\in\bbZ & \quad\Longleftrightarrow \quad & t\in\{\lambda_j\}_{j=1}^n,\\
&\varphi_h(t) =A_1+\pi m \mbox{ for some }  m\in\bbZ & \quad\Longleftrightarrow \quad & p(t)=0,\\
\label{eq:intersection2}
%&\varphi_h(t) =A_2+\pi r \mbox{ for some }  r\in\bbZ & \quad \Longleftrightarrow \quad & t\in\{\mu_j\}_{j=1}^{n-1}\cup\{0\}.
&\varphi_h(t) =A_2+\pi r \mbox{ for some }  r\in\bbZ & \quad \Longleftrightarrow \quad & tq(t)=0.
\end{alignat}

%By Lemma~\ref{lem:0}, $\varphi'_h\ne0$ at any of $\lambda_j$'s or $\mu_j$'s except at potentially $t=0$. In fact, from
%
We also want to understand the sign of the derivative of $\varphi_h$ at these points.  Straightforward calculations show that
\begin{equation*}
\varphi'_h(t)=\frac{d}{dt}  \left(\arctan \frac{\im h(t)}{\re h(t)} \right)=
\im\alpha\, \dfrac{t\left[p'(t)q(t)-p(t)q'(t)\right]-p(t)q(t)}{|h(t)|^2}.
\end{equation*}
Since $\im\alpha>0$, $p(t)q(t)=0$ at any zero of $p$ and $q$,
and $p'(t)q(t)-p(t)q'(t)>0$ for any $t\in\bbR$ (see, e.g.,~\cite[Theorem 3.4]{Holtz_Tyaglov}), we get
\begin{equation}\label{signs.1}
\sgn\, \varphi'_h(\lambda_j)=\sgn\,\lambda_j,\qquad\text{and}\qquad\sgn\, \varphi'_h(\mu_k)=\sgn\,\mu_k,
\end{equation}
for $j=1,\ldots,n$, $k=1,\ldots,n-1$. Notice also that
\begin{equation}\label{eq:at0}
\sgn\,\varphi_h(0) = -\sgn\, p(0)q(0).
\end{equation}

By~\eqref{limits} we have $\lim\limits_{t\to-\infty}\varphi_h(t)=-\pi(n_{+}-n_{-})$, where
$n_{+}$ and $n_{-}$ are the numbers of zeros of~$h$ in $\bbC_+$ and $\bbC_-$,
respectively.

Suppose first that $0<s<n$, that is, $\lambda_1<0$ and $\lambda_n>0$. Then $\varphi_h'(\lambda_1)<0$  by~\eqref{signs.1}, %, that is, $\varphi$ is decreasing at $\lambda_1$, which means
%Since $\lambda_1<\mu_1$, $A_1>A_2$, and $\varphi_h'(\lambda_1)<0$  by~\eqref{signs.1},
%the first time the hodograph of $h$ will intersect the line $\{re^{iA_1}\ :\ r\in\bbR\}$ before the line $\{re^{iA_2}\ :\ r\in\bbR\}$ with
so that~\eqref{eq:intersection1} implies $\varphi_{h}(\lambda_1)
%=-\pi(n_{+}-n_{-})-(\pi-A_1)
=A_1-\pi(n_{+}-n_{-}+1)$.
 By~\eqref{signs.1} $\varphi_h$ decreases at each of the points $\lambda_j$, $j=1,\ldots,s$,
and $\mu_k$, $k=1,\ldots,s-1$. Using~\eqref{eq:intersection1} and~\eqref{eq:intersection2} and induction, we conclude
\begin{equation}\label{values.1}
\varphi_{h}(\lambda_j)=A_1-\pi(n_{+}-n_{-}+j),\quad\varphi_{h}(\mu_k)=A_2-\pi(n_{+}-n_{-}+k),
\end{equation}
for $j=1,\ldots,s$, and $k=1,\ldots,s-1$.

On the interval $(\lambda_s,\lambda_{s+1})$, the polynomial $zq(z)$ has two zeros, $\mu_s$ and $0$.
If $\mu_s<0$, then $\varphi'_{h}(\mu_s)<0$ and $\varphi'_{h}(0)>0$, see~\eqref{signs.1} and~\eqref{eq:at0}.
Similarly, if $\mu_s>0$, then $\varphi'_{h}(0)<0$
and $\varphi'_{h}(\mu_s)>0$. If $\mu_s=0$, then $\varphi'_{h}(0)=0$. In any case \eqref{eq:intersection2} gives us
\begin{equation}\label{values.2}
\varphi_{h}(\mu_s)=\varphi_{h}(0)=A_2-\pi(n_{+}-n_{-}+s).
\end{equation}
From~\eqref{signs.1} it follows that at the points $\lambda_j$, $j=s+1,\ldots,n$,
and $\mu_k$, $k=s+1,\ldots,n-1$, the function $\varphi_h$ increases so that
\begin{equation}\label{values.3}
\varphi_{h}(\lambda_j)=A_1-\pi(n_{+}-n_{-}+2s+1-j),\quad\varphi_{h}(\mu_k)=A_2-\pi(n_{+}-n_{-}+2s+1-k),
\end{equation}
for $ j=s+1,\ldots,n$, and $k=s+1,\ldots,n-1$. In particular, we have
$$
\varphi_{h}(\lambda_n)=A_1-\pi(n_{+}-n_{-}+2s+1-n)
$$
and $\varphi'_h(\lambda_n)>0$. By \eqref{eq:intersection1}--\eqref{eq:intersection2}, we can conclude that for  $t>
\lambda_n$,
$$
A_1-\pi(n_{+}-n_{-}+2s+1-n)<\varphi_h(t)<\pi+A_2-\pi(n_{+}-n_{-}+2s+1-n).
$$
But $\lim\limits_{t\to+\infty}\varphi(t)=0$ by~\eqref{limits} {so that}
$$
\lim\limits_{t\to+\infty}\varphi(t)=\pi-\pi(n_{+}-n_{-}+2s+1-n)=
\pi(n-2s-n_{+}+n_{-})=0.
$$
This implies $n_{+}=n-s$ and $n_{-}=s$, {since $n=n_{+}+n_{-}$}.

Now from~\eqref{eq:a} it follows that
\begin{multline}\label{eq:a(0)}
\varphi_h(0) = \sum_{j=1}^n \Arg(-z_j)
= \sum_{z_j\in \bbC_+}   \big[(\Arg_{[0,\pi)} z_j ) -\pi\big]
+
\sum_{z_j\in \bbC_-}   \Arg_{[0,\pi)} z_j
\\
= \sum_{j=1}^n   (\Arg_{[0,\pi)} z_j ) -\pi n_+. %=A_2 -\pi n_+,
\end{multline}
By~\eqref{values.2} $\varphi_h(0)=A_2-\pi n_+$, which proves that zeros indeed belong to~\eqref{eq:zerosNEW}. %, and $s$ and $n-s$ are the numbers of zeros of $h$ in $\bbC_-$ and $\bbC_+$, respectively.

If $s=n$ then same arguments give us~\eqref{values.1} for all $j$ and $k$. Then $\varphi'_h(0)<0$ by~\eqref{eq:at0}, and $\varphi_{h}(0)=A_2-\pi(n_{+}-n_{-}+n)$. Using~\eqref{eq:intersection1}--\eqref{eq:intersection2}  we then have $A_2-\pi(n_{+}-n_{-}+n)>\varphi_h(t)>A_1-\pi-\pi(n_{+}-n_{-}+n)$ for $t>0$. Then we get $\lim\limits_{t\to+\infty}\varphi(t)=-\pi(n_{+}-n_{-}+n)=0$. This produces $n_+ = 0$, $n_- = n$, and then ~\eqref{eq:a(0)} leads to ~\eqref{eq:zerosNEW} as before.

Finally, if $s=0$ then by~\eqref{eq:intersection1}--\eqref{eq:intersection2} we get $\varphi_h'(0)>0$ and $\varphi_h(0)=A_2-\pi(n_{+}-n_{-})$. Then $\varphi'_h(\lambda_j)>0$, $\varphi'_h(\mu_k)>0$, and
$\varphi_{h}(\lambda_j)=A_1-\pi(n_{+}-n_{-}+1-j)$, $\varphi_{h}(\mu_k)=A_2-\pi(n_{+}-n_{-}-k)$ for all $j,k$. This implies $\varphi_{h}(\lambda_n)=A_1-\pi(n_{+}-n_{-}+1-n)$ and
$$
A_1-\pi(n_{+}-n_{-}+1-n)<\varphi_h(t)<A_2-\pi(n_{+}-n_{-}-n)
$$
for $t>\lambda_n$.  Then we get $\lim\limits_{t\to+\infty}\varphi(t)=-\pi(n_{+}-n_{-}-n)=0$. This produces $n_+ = n$, $n_- = 0$, and then ~\eqref{eq:a(0)} leads to ~\eqref{eq:zerosNEW} as before.

\medskip

$[\Leftarrow]$
Conversely, let the zeros of the polynomial $h$ belong to~\eqref{eq:zerosNEW} for a given complex
number~$\alpha$ with~$\im\alpha>0$. From~\eqref{eq:zerosNEW} we have
\begin{equation*}
\Arg\,\alpha=\Arg_{[0,\pi)} \alpha=\sum_{j=1}^n \Arg_{[0,\pi)} z_j
%=\big(\Arg \prod_{j=1}^n z_j\big) \operatorname{mod} \pi,
=\Arg_{[0,\pi)} \Big( \prod_{j=1}^n z_j\Big),
\end{equation*}
and therefore by~\eqref{eq:cotan}
\begin{equation}\label{eq:argalpha}
\frac{\re \alpha}{\im\alpha} = \frac{\re \prod\limits_{j=1}^n z_j}{\im \prod\limits_{j=1}^n z_j}.
\end{equation}

Let $h(z) = \sum\limits_{j=0}^n h_j z^j$, $p(z) = \sum\limits_{j=0}^n p_j z^j$,
$r(z) = \sum\limits_{j=0}^{n} r_j z^j$, where $p_j$'s and $r_j$'s
are to be determined from the condition
\begin{equation*}
h(z) = \alpha p(z) + (1-\alpha)r(z).
\end{equation*}
By equating $z^j$ coefficients, we get $h_j = \alpha p_j + (1-\alpha)r_j$ which
gives us the system of linear equations
\begin{equation}\label{lin.sys}
\begin{pmatrix}
\re \alpha & 1-\re \alpha \\
\im\alpha & -\im\alpha
\end{pmatrix}
\begin{pmatrix}
p_j \\ r_j
\end{pmatrix}
=
\begin{pmatrix}
\re h_j \\ \im h_j
\end{pmatrix}.
\end{equation}
Determinant of the coefficient matrix is equal to $-\im\alpha\ne 0$, so that $p_j$ and
$r_j$ are \textit{uniquely} determined. Since $h_n=1$ we have $p_n=r_n=1$ so that $p$ and
$r$ are monic of degree $n$. By~\eqref{eq:argalpha} and the fact that
$h_0=(-1)^n\prod\limits_{j=1}^n z_j$, the solution of the system~\eqref{lin.sys} for $j=0$
has the form
\begin{equation*}
p_0=
(-1)^n \frac{\im  \prod\limits_{j=1}^n z_j}{\im\alpha}\qquad\text{and}\qquad r_0=r(0)=0.
\end{equation*}
Thus, $r(z) = zq(z)$ for a polynomial $q$ of degree $n-1$, so
the formula~\eqref{eq:extHBpolyNEW} holds, where the polynomials $p$ and $q$ are uniquely
determined by the polynomial $h$ and the number $\alpha$.

Let us show that the zeros $\{\lambda_j\}_{j=1}^n$
 of $p$ and  $\{\mu_j\}_{j=1}^{n-1}$ of $q$ are real, simple and strictly interlacing.
 {To this end}, define $A_1$ and $A_2$ with $\pi>A_1>A_2>0$ as in~\eqref{numbers.A} and observe that~\eqref{eq:intersection1}--\eqref{eq:intersection2} still holds by the identical arguments (notice that $p(t)$ and $tq(t)$ cannot vanish simultaneously for $t\in\bbR$ since $h$ has no real roots).

%From~\eqref{eq:extHBpolyNEW} it follows that $\varphi_h(t)=A_1+\pi k$ for some $k\in\mathbb{Z}$ if and only if $p(t)=0$,
%and $\varphi_h(t)=A_2+\pi r$ for some $r\in\mathbb{Z}$ if and only if $tq(t)=0$, where the
%numbers $A_1$ and $A_2$ are defined in~\eqref{numbers.A}, $\pi>A_1>A_2>0$.

Note that ~\eqref{eq:a(0)} holds  where $n_+$ is the number of zeros of $h$ in $\bbC_+$. Combining this with~\eqref{eq:zerosNEW} and~\eqref{numbers.A}, we get $\varphi_h(0)=A_2-\pi n_+$. Moreover, $\lim\limits_{t\to-\infty}\varphi_h(t)
=\pi(n_{-}-n_{+})$ and $\lim\limits_{t\to+\infty}\varphi_h(t)=0$ by~\eqref{limits}.
Now the intermediate value theorem for continuous functions implies that  $\varphi_h$ has to attain each value of the form $A_1+\pi m$ ($m\in\bbZ$) in the range $\varphi_h(0)<y<\varphi_h(-\infty)$ at least once on $t<0$. There are $n_-$ of them: $A_1-(n_+-j)\pi$ with $0\le j\le n_- -1$ (this includes the case $n_-=0$ which is special since then $\varphi_h(0)>\varphi_h(-\infty)$).
% $A_1- n_+\pi,A_1-(n_+-1)\pi,\ldots, A_1-(n_+-n_-+1)\pi$
By~\eqref{eq:intersection1} this shows that there are at least~$n_-$ distinct negative zeros  of $p$. Similarly, there are $n_+$ of values of the form $A_1+k \pi$ in the range $\varphi_h(0)<y<\varphi_h(+\infty)$:  these are $A_1-(n_+-j)\pi$ with $0\le j\le n_+ -1$.
% $A_1- n_+\pi,A_1-(n_+-1)\pi,\ldots, A_1-\pi$
By the intermediate value theorem and ~\eqref{eq:intersection1} this determines at least $n_+$ distinct positive zeros of $p$. Since $n_-+ n_+ = n$, and $\deg p = n$, this implies that $p$ has only real simple zeros $\{\lambda_j\}_{j=1}^n$.

Similarly, see~\eqref{eq:intersection2}, the zeros of $zq(z)$ are determined by checking where $\varphi_h$ attains the values  $A_2+\pi r$ ($r\in\bbZ$).
There is at least one such value on each interval $(\lambda_j,\lambda_{j+1})$ for each $j\ne n_-$.

If $n_->0$ then we need to inspect the interval $(\lambda_{n_-},\lambda_{n_-+1})$ where we have $\varphi_h(0)=A_2-\pi n_+$ and $\varphi_h(\lambda_{n_-}) = \varphi_h(\lambda_{{n_-}+1}) >\varphi_h(0)$. This means that apart from $t=0$ there is at least one more zero of $tq(t)$ (potentially also at $t=0$ which is allowed).  Thus, we determined all $n$ zeros of $tq(t)$, and zeros of $q$ are shown to be strictly interlacing with the zeros of $p$.

If $n_-=0$ then we already have found $n-1$ of {\it positive} zeros of $tq(t)$ on each interval $(\lambda_j,\lambda_{j+1})\subset \bbR_+$. These are, therefore, $n-1$ zeros of $q$ which are all real, simple, and strictly interlacing with~$\{\lambda_j\}_{j=1}^n$.
\end{proof}

%{\color{blue}\textsc{I am not sure whether we really need the following remark.}}
%{\color{red}
%\begin{remark}\label{rem:maybeComment}
%	If we have   $\im\alpha<0$ instead, then the exact same statement of Theorem~\ref{thm:HBextendedNEW} holds but with the
%condition $\sum_{j=1}^n (\Arg\, z_j \operatorname{mod} \pi) =\Arg\,\alpha$ in ~\eqref{eq:zerosNEW} replaced by
%$\sum_{j=1}^n ((-\Arg\, z_j)  \operatorname{mod} \pi) =-\Arg\,\alpha $. In other words, all the zeros move to the
%second and fourth quadrants, and angles should be counted {\textnormal{clockwise}} from the positive $\bbR_+$ half-axis
%if $z\in\bbC_-$ and from the negative $\bbR_-$ half-axis if $z\in\bbC_+$.
%\end{remark}
%
%}

\section{Generalized Hermite--Biehler theorem: broken interlacing around $t=0$}

The next theorem is an analogue of Theorem~\ref{thm:HBextendedNEW} for the case when interlacing of two polynomials is still broken at one location around the origin, but we no longer  assume $t=0$ to be one of the zeros, see ~\eqref{almost.interlacing}. Another way to encode this is to say that the zeros of $zp(z)$ and $q(z)$ strictly interlace (warning: there is no $zp(z)$ in~\eqref{eq:extHBpolyNEW.2}!) This type of broken interlacing naturally appears for multiplicative perturbations of {\it singular} Jacobi matrices, see Section~\ref{ss:multipl} below.

On a related note, we observe that if we are in the setting of Theorem~\ref{thm:HBextendedNEW} but with $p(0)=0$, then $p(z) = z\widehat{p}(z)$ and therefore $h(z)=z\big[\alpha \widehat{p}(z) + (1-\alpha)q(z)\big]$,
and upon dividing by $z$ we end up in the setting of Theorem~\ref{rem:around0}.

%allows to generalize the result by replacing
%the polynomial $zq(z)$ with an arbitrary polynomial $q$ of degree $n$ which has interlacing . Namely, the following theorem holds.
%
\begin{theorem}\label{rem:around0}
Let $p$ and $r$ be monic real polynomials, $\deg p= \deg r = n$, and $\alpha\in\bbC_+$.
Define  the
monic polynomial
	\begin{equation}\label{eq:extHBpolyNEW.2}
		h(z) = \alpha p(z) + (1-\alpha) r(z) = \prod_{j=1}^n (z-z_j).
	\end{equation}
If the zeros $\{\lambda_j\}_{j=1}^n$ of $p$ and $\{\mu_j\}_{j=1}^n$ of $r$ %are real, simple, and %interlacing if and only if $h$ has all of its zeros
	satisfy the inequalities
%
%\begin{equation}\label{almost.interlacing}
%\lambda_1 < \mu_1 < \ldots \lambda_s < \mu_s\leqslant0 \leqslant \mu_{s+1} < \lambda_{s+1} < \ldots < \mu_n <\lambda_n
%\end{equation}
\begin{equation}\label{almost.interlacing}
\lambda_1 < \mu_1 < \ldots \lambda_s < \mu_s<0 < \mu_{s+1} < \lambda_{s+1} < \ldots < \mu_n <\lambda_n
\end{equation}
(for some $0\le s\le n-1$) , then
%Given a number $\alpha$, $\im\alpha>0$, and a monic complex polynomial $h(z)=\prod\limits_{j=1}^n(z-z_j)$,
the zeros of $h(z)$ belong to
\begin{equation}\label{eq:zerosMORE}
\left\{ \{z_j\}_{j=1}^n \in (\bbC\setminus\bbR)^n: \sum_{j=1}^n \Arg_{[0,\pi)}
z_j < \Arg\,\alpha\right\}.
\end{equation}

Conversely, if  $\alpha\in\bbC_+$ and for the zeros of a monic complex polynomial $h(z) = \prod_{j=1}^n (z-z_j)$
the condition~\eqref{eq:zerosMORE} holds, then there exists a unique pair of monic real polynomials $p$ and $r$ with strictly interlacing zeroes such that $h(z) = \alpha p(z) + (1-\alpha) r(z) $.

%if and only if there exists a unique pair of monic real polynomials
%
%\begin{equation}\label{polys}
%p(z)=\prod\limits_{j=1}^n(z-\lambda_j)\qquad\text{and}\qquad q=\prod\limits_{j=1}^{n}(z-\mu_j),
%\end{equation}
%
%such that
%
%\begin{equation}\label{eq:extHBpolyNEW.2}
%h(z) = \alpha p(z) + (1-\alpha)q(z),
%\end{equation}
%
%and the zeros of $p$ and $q$ satisfy the inequalities
%
%\begin{equation}\label{almost.interlacing}
%\lambda_1 < \mu_1 < \ldots \lambda_s < \mu_s\leqslant0 \leqslant \mu_{s+1} < \lambda_{s+1} < \ldots < \mu_n <\lambda_n
%\end{equation}
%

The number of $z_j$'s in $\bbC_+$ (respectively, in $\bbC_-$) coincides with the
number of positive (respectively, negative)  zeros of $p$.
\end{theorem}
\begin{remark}
Compare with Theorem \ref{thm:HBextendedNEW}: if $\mu_s \to 0$ or $\mu_{s+1}\to 0$ then the zeros $\{z_j\}_{j=1}^n$ approach~\eqref{eq:zerosNEW}.
\end{remark}
\begin{remark}
By~\eqref{almost.interlacing} with $s=0$ we mean
\begin{equation}\label{almost.interlacing2}
0 < \mu_{1} < \lambda_{1} < \ldots < \mu_n <\lambda_n.
\end{equation}
Compare with
Proposition~\ref{rem:HBlocalization}(i) with $s=0$.
\end{remark}

\begin{proof}

$[\Rightarrow]$
%First we note that the three cases $\mu_s=\mu_{s+1}=0$; $\mu_s<0=\mu_{s+1}$; $\mu_s=0<\mu_{s+1}=0$ are covered
%by Theorem~\ref{thm:HBextendedNEW}: we then get the equality sign in~\eqref{eq:zerosMORE}. So let us assume that none of $\mu_j$'s are zero and show that we get zeros of in ~\eqref{eq:zerosMORE} with the strict inequality sign.
Given $\alpha\in\bbC_+$, suppose that the zeros of the polynomials $p$ and $r$ satisfy
the inequalities~\eqref{almost.interlacing}, and consider the polynomial $h$
%$h(z)=\prod\limits_{j=1}^n(z-z_j)$
defined in~\eqref{eq:extHBpolyNEW.2}. Repeating the arguments in the $[\Rightarrow]$ proof of Theorem~\ref{thm:HBextendedNEW} with $r(t)$ instead of $tq(t)$ and $\{\mu_j\}_{j=1}^n$ instead of $\{0\}\cup\{\mu_j\}_{j=1}^{n-1}$, we get that the
function~$\varphi_h$ defined in~\eqref{eq:a} is well defined on~$\bbR$ (since $p$ and {$r$} do not vanish on~$\bbR$
simultaneously) and satisfies the identities~\eqref{signs.1},~\eqref{values.1} and~\eqref{values.3}. The analogue of~\eqref{values.2} is $\varphi_{h}(\mu_s)=\varphi_{h}(\mu_{s+1})=A_2-\pi(n_{+}-n_{-}+s)$. Because $\varphi'_{h}(\mu_s)<0$ and $\varphi'_{h}(\mu_{s+1})>0$ we arrive to $\varphi_h(0)< A_2-(n_{+}-n_{-}+s)\pi$. Then using $s=n_-$, $n-s=n_+$ (the same arguments as in the proof of Theorem~\ref{thm:HBextendedNEW}) and ~\eqref{eq:a(0)}, we get $\sum\limits_{j=1}^n \Arg_{[0,\pi)}z_j < \Arg\,\alpha$.

%$\varphi_{h}(\mu_s)=\varphi_{h}(\mu_{s+1})=A_2-\pi(n_{+}-n_{-}+s)$ with $\varphi'_{h}(\mu_s)<0$ and $\varphi'_{h}(\mu_{s+1})>0$,
%so that %$A_1 - (n-s-1)\pi< \varphi_h(0)\leqslant A_2-(n-s)\pi$, and now~\eqref{eq:a(0)} implies~\eqref{eq:zerosMORE}.
%$\varphi_h(0)< A_2-(n_{+}-n_{-}+s)\pi$. Then~\eqref{eq:a(0)} together with $s=n_-$, $n-s=n_+$ imply that $z_j$'s belong to implies~\eqref{eq:zerosMORE}.

$[\Leftarrow]$ Conversely, let the zeros of the polynomial $h$ belong to~\eqref{eq:zerosMORE} for a given complex
number $\alpha\in\bbC_+$. Then from~\eqref{eq:a(0)} and~\eqref{eq:zerosMORE} we have
\begin{equation}\label{ineq.1}
\varphi_h(0)< A_2-\pi n_{+}<A_1-\pi n_{+},
\end{equation}
where $A_1$ and $A_2$ are defined in~\eqref{numbers.A}, and $n_{+}$ is the number of zeros of $h$ in the
$\bbC_+$. In the same way as in the proof of Theorem~\ref{thm:HBextendedNEW}, one can prove that
there exists a \textit{unique} pair of real monic polynomials $p$ and $r$ of degree $n$ such that
\eqref{eq:extHBpolyNEW.2} holds. Following the proof of Theorem~\ref{thm:HBextendedNEW} and
using~\eqref{ineq.1}, one can show that the polynomial $p$ has real and simple zeros $\lambda_j$'s satisfying
\begin{equation*}
\lambda_{1}<\ldots<\lambda_{n_{-}}<0<\lambda_{n_{-}+1}<\ldots<\lambda_n,
\end{equation*}
where $n_{-}$ is the number of zeros of $h$ in $\bbC_-$, $n_{-}+n_{+}=n$ while
the polynomial $r(z)=\prod_{j=1}^n (z-\mu_j)$ has an odd number of zeros, counting multiplicities, on every interval $(\lambda_j,\lambda_{j+1})$,
$j\ne {n_-}$ and an even number of zeros (at least two), counting multiplicities, on $(\lambda_{n_{-}},\lambda_{n_{-}+1})$
due to $\varphi_h(\lambda_{n_-}) = \varphi_h(\lambda_{{n_-}+1}) > \varphi_h(0)$. Moreover, because of $\varphi_h(0)< A_2-\pi n_{+}$, we must have two of $\mu_j$'s distinct from $0$. This proves~\eqref{almost.interlacing}.
\end{proof}

\begin{remark}\label{rem:shifted}
We can restate Theorem~\ref{thm:HBextendedNEW} with the shift from $t=0$ to an arbitrary chosen
$t=\xi\in\bbR$. Then $h$ in~\eqref{eq:extHBpolyNEW} should be replaced with
$h(z) = \alpha p(z) + (1-\alpha) (z-\xi) q(z)$ while~\eqref{eq:zerosNEW} should be replaced with
\begin{equation*}
\left\{ \{z_j\}_{j=1}^n \in (\bbC\setminus\bbR)^n: \sum_{j=1}^n \Arg_{[0,\pi)} (z_j-\xi)  =\Arg\,\alpha \right\}.
\end{equation*}
Similar shift can be done in Theorem~\ref{rem:around0} with $\mu_s<\xi<\mu_{s+1}$ in place of $\mu_s<0<\mu_{s+1}$ in~\eqref{almost.interlacing}.
\end{remark}

From~\eqref{eq:a(0)} it follows that
\begin{equation*}
\sum\limits_{j=1}^{n}\Arg\,z_j=\sum\limits_{j=1}^{n}\Arg\,(-z_j)+\pi(n_{+}-n_{-})=A_2-\pi n_{-}
\end{equation*}
This formula and Theorems~\ref{thm:HBextendedNEW} and~\ref{rem:around0} imply the following
analogues of Proposition~\ref{rem:HBlocalization}.
\begin{proposition}\label{prop:loc1}
In the setting of Theorem~\ref{thm:HBextendedNEW}, $p$ has $s$ negative zeros and $n-s$ positive ones if and only if the
polynomial $h(z)=\prod\limits_{j=1}^n(z-z_j)$ has $s$ zeros in $\bbC_-$, $n-s$ zeros in $\bbC_+$,
and
\begin{equation*}%\label{eq:loc1}
\sum_{j=1}^n \Arg\, z_j  =\Arg\,\alpha-\pi s.
\end{equation*}
\end{proposition}
\begin{proposition}\label{prop:loc2}
In the setting of Theorem~\ref{rem:around0}, $p$ has $s$ negative zeros and $n-s$ positive if and only if the
polynomial $h(z)=\prod\limits_{j=1}^n(z-z_j)$ has $s$ zeros in $\bbC_-$, $n-s$ zeros in $\bbC_+$,
and
\begin{equation*}%\label{eq:loc2}
-\pi s<\sum_{j=1}^n \Arg\, z_j < \Arg\,\alpha-\pi s.
\end{equation*}
\end{proposition}

%%%%%%%%%%%%%%%%%%%%%%%%%%%%%%%%%%%%%%%%%%%%%%%%%%%%%%%
\section{Applications}\label{s:applications}
%%%%%%%%%%%%%%%%%%%%%%%%%%%%%%%%%%%%%%%%%%%%%%%%%%%%%%%

In this section, we show how the Hermite-Biehler theorem and its extensions
from the previous sections can be applied to solve the direct and inverse spectral
problems for additive and multiplicative rank-one (and some rank-two) perturbations of
%certain classes
Hermitian %and tridiagonal
matrices.

Let
\begin{equation}\label{eq:jacobi}
\calJ =
\begin{pmatrix}
b_1&a_1&0& &\\
a_1&b_2&a_2&\ddots &\\
0&a_2&b_3&\ddots & 0 \\
&\ddots&\ddots&\ddots & a_{n-1} \\
& & 0 & a_{n-1} & b_n
\end{pmatrix}, \quad a_j > 0,\ \  b_k\in\bbR.
\end{equation}
For each such $\calJ$, let us also define $\calJ^{(j)}$ to be the $(n-j)\times (n-j)$ Jacobi matrix obtained from~\eqref{eq:jacobi}
by removing its first $j$ rows and first $j$ columns. Let $p_j$, $j=1,\ldots,n-1$, be the characteristic polynomials of
$\calJ^{(j)}$, respectively, and $p_0$ is the characteristic polynomial of $\calJ$. It is easy to check that
\begin{equation}\label{rec.rel}
p_{j-1}(z)=(z-b_{j})p_{j}(z)-a_{j}^2p_{j+1}(z),\qquad j=1,\ldots,n,
\end{equation}
where $p_{n}(z)\equiv1$ and $p_{n+1}(z)\equiv1$. It is well-known that the zeros of $p_0(z)$ and $p_1(z)$ are real, simple, and strictly
interlacing~\cite{KreinGantmacher}. As was established by Gray and Wilson~\cite{GrayWilson}
and independently by Hald~\cite{Hald} (see also Hochstadt~\cite{Hochstadt}), given two interlacing sets of real numbers $\{\lambda_j\}_{j=1}^n$ and $\{\mu_k\}_{k=1}^{n-1}$
%satisfying the inequalities
%
%\begin{equation}\label{strict.interlacing}
%\lambda_1<\mu_1<\lambda_2<\cdots<\lambda_{n-1}<\mu_{n-1}<\lambda_n,
%\end{equation}
%%
there exists \textit{a unique} matrix $\calJ$ defined in~\eqref{eq:jacobi} such that these two sets are the eigenvalues of
$\calJ$ and $\calJ^{(1)}$, respectively. %de Boor and Golub~\cite{deBoor.Golub} found a numerically stable algorithm to reconstruct~$\calJ$ from the sets $\{\lambda_j\}_{j=1}^n$ and $\{\mu_k\}_{k=1}^{n-1}$ satisfying~\eqref{strict.interlacing}.

%%%%%%%%%%%%%%%%%%%%%%%%%%%%%%%%%%%%%%%%%%%%%%%%%%%%%%%%%%%%%%%%%%%%%%%%%%%%%%
\subsection{Additive non-Hermitian rank one perturbations}\label{ss:additive}
%%%%%%%%%%%%%%%%%%%%%%%%%%%%%%%%%%%%%%%%%%%%%%%%%%%%%%%%%%%%%%%%%%%%%%%%%%%%%%

Let $e_1$ be the $n\times 1$ vector with the first entry equal to $1$
and $0$ everywhere else. %$I_{1\times 1}$ be the $n\times n$ matrix with $(1,1)$-entry equal to $1$ and $0$ everywhere else.
Consider the non-Hermitian (additive) perturbation
\begin{equation}\label{eq:additiveJ}
\calJ_{l,+}=\calJ + i l e_1 e_1^*
%I_{1\times 1}
\end{equation}
of $\calJ$. In other words, $\calJ_{l,+}$ is~\eqref{eq:jacobi} but with $(1,1)$-entry
replaced by $b_1+il$. See~\cite{AT} for an in-depth study of such perturbations,
and~\cite{Wall} and~\cite[Chapter~X]{Wall.b} for a  related topic in theory of stable polynomials.

Let $h(z)$ be the characteristic polynomial $h(z)$ of $\calJ_{l,+}$. Using the
Laplace expansion for the determinants we obtain the following representation for $h(z)$
\begin{equation*}%\label{char.poly.1}
h(z)=(z-b_1-il)p_1(z) - a_1^2 p_2(z)=p_0(z) - ilp_1(z).
\end{equation*}
As  mentioned above, the zeros of $p_0$ and $p_1$ are real, simple, and strictly interlace.
Therefore, by the Hermite--Biehler theorem, Theorem~\ref{thm:HB}, the eigenvalues of $\calJ_{l,+}$
lie in the open upper (lower) half-plane whenever $l>0$ ($l<0$).

The inverse problem for matrix $\calJ_{l,+}$ requires to find the matrix from its spectrum.
Namely, let the numbers $\{z_j\}_{j=1}^{n}$ lie in $\bbC_+$. Then the polynomial
$h(z)=\prod\limits_{j=1}^{n} (z-z_j)$ can be \textit{uniquely} represented as $h=p-ilq$ for some $l\in\bbR$ and
 monic real polynomials $p$ and $q$ with $\deg p=n$, $\deg q=n-1$. From Theorem~\ref{thm:HB} it
follows that $l>0$, and the zeros of $p$ and $q$ are real, simple, and strictly interlace.
These $p$ and $q$ allow to uniquely recover $\calJ$, and the number $l$ determines the perturbation magnitude, see~\eqref{eq:additiveJ}.
%Then matrix $\calJ_{l,+}$ is reconstructed by~\eqref{eq:additiveJ}.
Thus, $\calJ_{l,+}$~\eqref{eq:additiveJ}
can be reconstructed from its spectrum, and the solution of the inverse problem is {unique}.

\smallskip

This also allows us to study the spectrum of rank-one perturbations of generic Hermitian matrices: recall that any $n\times n$ Hermitian matrix $H=H^*$
with a cyclic vector $\bm{v}$ can be reduced to a Jacobi form $\calJ$~\eqref{eq:jacobi} via $H=S^*\calJ S$ for some unitary $S$ with $S\bm{v} =  \bm{e}_1$ (the Lanczos algorithm~\cite{Lanczos}). Using this, one can reduce any
\begin{equation}\label{eq:additiveH}
%$
 	 H  + i \Gamma
%$
\end{equation}
with $H=H^*$, $\Gamma=\Gamma^*$, $\rank\,\Gamma = 1$,
to~\eqref{eq:additiveJ}. Indeed, $\Gamma=\Gamma^*$ with $\rank\,\Gamma = 1$ imply $\Gamma = l \bm{v} \bm{v}^*$ for some $\bm{v}\in\bbC^n\setminus{\bm{0}}$, $l\in\bbR\setminus\{0\}$. If $\bm{v}$ is cyclic for $H$ then \begin{equation}\label{eq:Lan}
S(H + i \Gamma) S^* = \calJ+ il S\bm{v}\bm{v}^* S^*  = \calJ_{l,+}
\end{equation}
and we end up in the setting above.

%\marginpar{to R: {\color{red}Shall we explain the underlined? to M: I added more details}}

Moreover, assuming that signature of $H$ is given, one ends up in the setting of Proposition~\ref{rem:HBlocalization} where we know the number of
 positive/negative zeros of $p$. In particular, if $\bm{v}$ is cyclic and $l>0$ then eigenvalues of ~\eqref{eq:additiveH} belong to~\eqref{eq:HBlocalization2} or to~\eqref{eq:HBlocalization}, depending on whether $\lambda=0$ is an eigenvalue of $H$ or not, respectively.

If instead $\bm{v}$ is not cyclic, i.e., $\dim \operatorname{span}\{H^j \bm{v}:j\ge0\} = k<n$ then the Gram--Schdmidt procedure in the Lanczos algorithm terminates early. We still get~\eqref{eq:Lan} but now with $a_k=0$ in ~\eqref{eq:jacobi}, and so we get $k$ eigenvalues in $\bbC_+$ and $n-k$ real eigenvalues.

See~\cite{AlpK21,K17} for applications of this to additive perturbations of random matrices.

%%%%%%%%%%%%%%%%%%%%%%%%%%%%%%%%%%%%%%%%%%%%%%%%%%%%%%%%%%%%%%%%%%%%%%%%%%%%%%%%%%%%
\subsection{Multiplicative non-Hermitian  rank one perturbations}\label{ss:multipl}
%%%%%%%%%%%%%%%%%%%%%%%%%%%%%%%%%%%%%%%%%%%%%%%%%%%%%%%%%%%%%%%%%%%%%%%%%%%%%%%%%%%%

Let us now consider {\it multiplicative} rank-one perturbation
\begin{equation}\label{eq:multJ}
\calJ_{k,\times} = \calJ (I_n+i k I_{1\times 1})
\end{equation}
with $k>0$. Here $I_n$ is the $n\times n$ identity matrix. In other words, $\calJ_{k,\times}$ is~\eqref{eq:jacobi} but
with $(1,1)$-entry being replaced by $b_1(1+ik)$ and $(2,1)$-entry replaced by $a_1(1+ik)$.

Let us find the location of the spectrum of $\calJ_{k,\times}$. To this end, we find its characteristic polynomial~$h$:
\begin{equation*}
h(z) = \det(z-\calJ_{k,\times}) = [z-(1+ik)b_1]p_1(z) - a_1^2 (1+ik) p_2(z).
\end{equation*}
Now since $a_1^2 p_2(z) = (z-b_1) p_1(z) - p_0(z)$ by~\eqref{rec.rel}, {one gets}
\begin{equation}\label{char.poly.2}
h(z) = (1+ik) p_0(z) - ik zp_1(z).
\end{equation}

In the next two statements we solve the direct and inverse spectral problems for $\calJ_{k,\times}$. The cases $\det\calJ\ne 0$ and $\det\calJ=0$ shall be treated separately.

\begin{corollary}\label{cor:mult1}
Suppose $\det\calJ\ne 0$. Then the spectrum of $\calJ_{k,\times}$ (\eqref{eq:multJ}  with $k>0$) belongs to
\begin{equation}\label{eq:NEW1}
\left\{ \{z_j\}_{j=1}^n \in (\bbC\setminus\bbR)^n: \sum_{j=1}^n \Arg_{[0,\pi)}
z_j<\frac\pi2 \right\}.
\end{equation}
Conversely, each configuration of points from~\eqref{eq:NEW1} occurs as a spectrum of a unique $\calJ_{k,\times}$ with some $k>0$.

The number of positive and negative eigenvalues of $\calJ$ coincides with the number of eigenvalues of $\calJ_{k,\times}$
in $\bbC_+$ and $\bbC_-$, respectively.
\end{corollary}
\begin{proof}
The characteristic polynomial~$h$ of $\calJ_{k,\times}$ is~\eqref{char.poly.2}, i.e., we have~\eqref{eq:extHBpolyNEW} with $\alpha = 1+ik$. Clearly, as $k$ varies in $(0,\infty)$ we get
$\Arg\,\alpha \in(0,\tfrac\pi2)$, and each such value of $\Arg\,\alpha$ is achieved exactly once.
The result now follows from Theorem~\ref{thm:HBextendedNEW}.

Conversely, given a configuration of points $\{z_j\}_{j=1}^n$ in~\eqref{eq:NEW1}, first let $\alpha=1+ik$ where $k=\tan\big(\sum \Arg_{[0,\pi)}
z_j\big)>0$, so that $\alpha\in\bbC_+$ with $\Arg \,\alpha = \sum \Arg_{[0,\pi)}
z_j$. By Theorem ~\ref{thm:HBextendedNEW}, the set $\{z_j\}_{j=1}^n$ is the zero set of a unique polynomial~$h$ of the form~\eqref{char.poly.2} with
two polynomials $p_0$ and $p_1$ of degrees $n$ and $n-1$, respectively, with strictly interlacing zeros. As discussed above, $p_0$ and $p_1$ 
uniquely determine $\calJ$ such that $p_0$ and $p_1$ are the characteristic polynomials of $\calJ$ and $\calJ^{(1)}$. The uniqueness and existence 
of $\calJ$ and $k>0$  implies the uniqueness and existence of $\calJ_{k,\times}$ whose spectrum is $\{z_j\}_{j=1}^n$.
\end{proof}

For Jacobi matrices with $\det\calJ = 0$ we instead get the following result.

\begin{corollary}\label{cor:mult2}
Suppose $\det\calJ=0$. Then the spectrum of $\calJ_{k,\times}$ (\eqref{eq:multJ}  with $k>0$) contains a simple eigenvalue $0$, while the remaining eigenvalues belong to
\begin{equation}\label{eq:NEW2}
\left\{ \{z_j\}_{j=1}^{n-1} \in (\bbC\setminus\bbR)^{n-1}: \sum_{j=1}^{n-1} \Arg_{[0,\pi)}
z_j < \arctan k \right\}.
\end{equation}
Conversely, for any $k>0$, each such a configuration of points occurs as a spectrum of a unique~$\calJ_{k,\times}$ with $\det\calJ =0$.

The number of positive and negative eigenvalues of $\calJ$ coincides with the number of eigenvalues of $\calJ_{k,\times}$
in $\bbC_+$ and $\bbC_-$, respectively.
\end{corollary}
\begin{remark}
Unlike the situation in Corollary~\ref{cor:mult1}, if $k>0$ is arbitrary then the uniqueness for the inverse 
spectral problem for $\calJ_{k,\times}$ does not hold. Indeed, given a set of non-real numbers $\{z_j\}_{j=1}^{n-1}$, there are infinitely many 
numbers $k>0$ satisfying~\eqref{eq:NEW2}, and consequently, there are infinitely many matrices $\calJ_{k,\times}$ with the 
same spectrum (they each have $\det \calJ=0$ and distinct $k>0$).
\end{remark}
\begin{proof}
Since $\det\calJ=0$ we have  $p_0(z)=z\widehat{p}_0(z)$,
so that the characteristic polynomial of $\calJ_{k,\times}$ (see~\eqref{char.poly.2}) becomes
\begin{equation*}%\label{char.poly.3}
h(z) =z[(1+ik) \widehat{p}_0(z) - ik p_1(z)].
\end{equation*}
So if
\begin{equation*}%\label{polys}
\widehat{p}_0(z)=\prod\limits_{j=1}^{n-1}(z-\lambda_j)\qquad\text{and}\qquad p_1(z)=\prod\limits_{j=1}^{n-1}(z-\mu_j),
\end{equation*}
then the zeros $\{\lambda_j\}_{j=1}^{n-1}$ of $\widehat{p}_0$ and $\{\mu_j\}_{j=1}^{n-1}$ of $p_1$ satisfy the inequalities
\begin{equation}\label{almost.interlacing.2}
\lambda_1 < \mu_1 < \ldots \lambda_s < \mu_s<0 < \mu_{s+1} < \lambda_{s+1} < \ldots < \mu_{n-1} <\lambda_{n-1}
\end{equation}
for some integer $s$, $0\leqslant s\leqslant n-1$. As $\Arg(1+ik) = \arctan k$, by Theorem~\ref{rem:around0} we get that the $n-1$ nonzero eigenvalues of $\calJ_{k,\times}$ belong to~\eqref{eq:NEW2}.

Conversely, given $k>0$ and a configuration of points $\{z_j\}_{j=1}^{n-1}$ from ~\eqref{eq:NEW2} one applies Theorem~\ref{rem:around0} (with $\alpha:=1+ik$) to get two polynomials $p$ and $r$ of degree $n-1$ whose zeros $\{\lambda_j\}_{j=1}^{n-1}$ and $\{\mu_j\}_{j=1}^{n-1}$ satisfy ~\eqref{almost.interlacing.2}.
Then $z p(z)$ and $r(z)$ have strictly interlacing zeros. This uniquely determines $\calJ$ such that $zp(z)$ and $r(z)$ are the characteristic polynomials of $\calJ$ and $\calJ^{(1)}$.
\end{proof}

\smallskip

One can use this to study eigenvalues of multiplicative rank 1 perturbations of generic Hermitian matrices:
\begin{equation}\label{eq:multH}
	H (I + i \Gamma)
\end{equation}
where $H=H^*$, $\Gamma=\Gamma^*$, $\rank \, \Gamma = 1$:  assuming that a non-zero vector from $\ran \,\Gamma$ is cyclic for~$H$, the
matrix~\eqref{eq:multH} can be reduced to~\eqref{eq:multJ} via a unitary conjugation.

See~\cite{AlpK2} for an application of this to multiplicative perturbations of random matrices.

%%%%%%%%%%%%%%%%%%%%%%%%%%%%%%%%%%%%%%%%%%%%%%%%%%%%%%%%%%%
\subsection{Additive non-Hermitian rank-two perturbations}\label{ss:ranktwo}
%%%%%%%%%%%%%%%%%%%%%%%%%%%%%%%%%%%%%%%%%%%%%%%%%%%%%%%%%%%
Consider now an additive %rank-$2$ 
perturbation of $\calJ$ %the Jacobi matrix~\eqref{eq:jacobi}
of the following type
\begin{equation}\label{eq:jacobiRank2}
\calJ_{l,m} :=
\begin{pmatrix}
		b_1+il&a_1&0& &\\
		a_1+im&b_2&a_2&\ddots &\\
		0&a_2&b_3&\ddots & 0 \\
		&\ddots&\ddots&\ddots & a_{n-1} \\
		& & 0 & a_{n-1} & b_n
\end{pmatrix},
\end{equation}
with $m>0$ and $l\in\bbR$. A similar calculation as in the previous section shows that the characteristic polynomial of $\calJ_{l,m}$ is
\begin{equation}\label{char.poly.23}
h(z) = \left(1+\dfrac{im}{a_1} \right)p(z) -\dfrac{im}{a_1}  \left(z-\dfrac{mb_1-la_1}{m} \right)q(z).
\end{equation}
So it can be represented in the form
$$
\alpha p(z) + (1-\alpha) (z-\xi)q(z)
$$
with $\alpha = 1+\tfrac{im}{a_1}$ and
$\xi = b_1-\tfrac{la_1}{m}$. 

Thus, as above we can solve the direct and inverse spectral problems for $\calJ_{l,m}$. The cases $\det(\calJ-\xi I)\ne 0$ and $\det(\calJ-\xi I)=0$ must be treated separately.
\begin{theorem}\label{thm:2add1}
Given the numbers $m>0$ and $l\in\mathbb{R}$, the spectrum of the matrix $\calJ_{l,m}$ defined in~\eqref{eq:jacobiRank2} belongs to
\begin{equation}\label{eq:NEW11}
\left\{ \{z_j\}_{j=1}^n \in (\bbC\setminus\bbR)^n: \sum_{j=1}^n \Arg_{[0,\pi)}
(z_j-\xi)<\frac\pi2 \right\}
\end{equation}
with $\xi = b_1-\tfrac{la_1}{m}$ provided $\det(\calJ-\xi I)\ne 0$ where $\calJ=\calJ_{0,0}$.

\vspace{1mm}

Conversely, given a number $\xi\in\mathbb{R}$, each configuration of points from~\eqref{eq:NEW11} occurs as a spectrum 
of a unique matrix $\calJ_{l,m}$ with some $m>0$ and $l\in\mathbb{R}$.

The number of eigenvalues of $\calJ$ greater than $\xi$ and less than $\xi$  coincides with the number of eigenvalues of $\calJ_{l,m}$
in $\bbC_+$ and $\bbC_-$, respectively.
\end{theorem}
\begin{remark}
    Notice that if one does not fix $\xi$, then there is no hope of uniqueness as is clear from the parameter counting.
\end{remark}
\begin{proof}
The characteristic polynomial~$h$ of $\calJ_{l,m}$ is~\eqref{char.poly.23} with strictly interlacing polynomials $p$ and $q$, i.e., we are in the situation of 
Remark~\ref{rem:shifted} with $\alpha = 1+\tfrac{im}{a_1}$. Clearly, for a fixed $\xi$, as $m$ varies in $(0,\infty)$ we get
$\Arg\,\alpha=\arctan\dfrac{m}{a_1} \in\left(0,\dfrac\pi2\right)$, and each such value of $\Arg\,\alpha$ is achieved exactly once.
The result now follows from Theorem~\ref{thm:HBextendedNEW} and Remark~\ref{rem:shifted} (notice that $p(z)$ and $(z-\xi)q(z)$ has no common zero by our assumption).

\vspace{2mm}

Conversely, given $\xi\in\mathbb{R}$ and a configuration of points $\{z_j\}_{j=1}^n$ in~\eqref{eq:NEW11}, let 
$\alpha=1+i\tan A$ where $A=\sum\limits_{j=1}^n \Arg_{[0,\pi)} (z_j-\xi)\in\left(0,\dfrac\pi2\right)$, so that 
$\alpha\in\bbC_+$ with $\Arg \,\alpha = \sum\limits_{j=1}^n \Arg_{[0,\pi)} (z_j-\xi)$. By Theorem~\ref{thm:HBextendedNEW} and Remark~\ref{rem:shifted}, 
the set $\{z_j\}_{j=1}^n$ is the zero set of a unique polynomial~$h$ of the form~\eqref{char.poly.23} 
with two polynomials $p$ and $q$ of degrees $n$ and $n-1$, respectively, with strictly interlacing zeros. As 
discussed above, $p$ and $q$ uniquely determine $\calJ$ such that $p$ and $p$ are the characteristic 
polynomials of $\calJ$ and $\calJ^{(1)}$. In particular, they uniquely determine the numbers $a_1>0$ and $b_1\in\mathbb{R}$,
so that we define
\begin{equation}\label{coeff.l.m}
m:=a_1\cdot\tan A\qquad\text{and}\qquad l:=\dfrac{m}{a_1}(b_1-\xi).
\end{equation}
The uniqueness and existence of $\calJ$, $m>0$, and $l\in\mathbb{R}$ implies the uniqueness 
and existence of $\calJ_{l,m}$ whose spectrum is $\{z_j\}_{j=1}^n$.
\end{proof}

For Jacobi matrices with $\det(\calJ-\xi I) = 0$ (with $\xi = b_1-\tfrac{la_1}{m}$ with given $m>0$ and $l\in\mathbb{R}$) we instead get the following result.

\begin{theorem}\label{thm:2add2}
Given the numbers $m>0$ and $l\in\mathbb{R}$, suppose $\det(\calJ-\xi)=0$ where $\xi = b_1-\tfrac{la_1}{m}$. Then the spectrum 
of $\calJ_{l,m}$ defined in~\eqref{eq:jacobiRank2} contains a simple real eigenvalue $\xi$, while the remaining eigenvalues belong to
\begin{equation}\label{eq:NEW22}
\left\{ \{z_j\}_{j=1}^{n-1} \in (\bbC\setminus\bbR)^{n-1}: \sum_{j=1}^{n-1} \Arg_{[0,\pi)}
(z_j-\xi) < \arctan A \right\}
\end{equation}
where $A=\dfrac{m}{a_1}>0$.

\vspace{2mm}

Conversely, given $\xi\in\mathbb{R}$, for any $A>0$, each such a configuration of points occurs as a spectrum of a 
unique~$\calJ_{l,m}$ with $\det(\calJ-\xi) =0$, and $m>0$, $l\in\bbR$.

The number of eigenvalues of $\calJ$ greater than $\xi$  and less than $\xi$  coincides with the number of eigenvalues of $\calJ_{l,m}$
in $\bbC_+$ and $\bbC_-$, respectively.
\end{theorem}
%
%\begin{remark}
%Unlike the situation in Theorem~\ref{thm:2add1}, if $A>0$ is arbitrary (that is, not fixed) then the uniqueness for the inverse
%spectral problem for $\calJ_{m,l}$ does not hold. Indeed, given a set of non-real numbers $\{z_j\}_{j=1}^{n-1}$, there are infinitely many
%numbers $A>0$ satisfying~\eqref{eq:NEW22}, and consequently, there are infinitely many matrices $\calJ_{l,m}$ with the
%same spectrum (they each have $\det(\calJ-\xi)=0$ and distinct $m>0$ and $l\in\mathbb{R}$).
%\end{remark}
%
\begin{proof}
Since $\det(\calJ-\xi)=0$ we have  $p(z)=(z-\xi)\widehat{p}(z)$,
so that the characteristic polynomial of $\calJ_{k,\times}$ (see~\eqref{char.poly.2}) becomes
\begin{equation*}%\label{char.poly.3}
h(z) =(z-\xi)\left[\left(1+i\dfrac{m}{a_1}\right) \widehat{p}(z) - i\dfrac{m}{a_1} q(z)\right].
\end{equation*}
Zeros of $p$ and $q$ strictly interlace. So if
\begin{equation*}%\label{polys}
\widehat{p}(z)=\prod\limits_{j=1}^{n-1}(z-\lambda_j)\qquad\text{and}\qquad q(z)=\prod\limits_{j=1}^{n-1}(z-\mu_j),
\end{equation*}
then the zeros $\{\lambda_j\}_{j=1}^{n-1}$ of $\widehat{p}$ and $\{\mu_j\}_{j=1}^{n-1}$ of $q$ satisfy the inequalities
\begin{equation}\label{almost.interlacing.22}
\lambda_1 < \mu_1 < \ldots \lambda_s < \mu_s<\xi < \mu_{s+1} < \lambda_{s+1} < \ldots < \mu_{n-1} <\lambda_{n-1}
\end{equation}
for some integer $s$, $0\leqslant s\leqslant n-1$. As $\Arg\left(1+i\dfrac{m}{a_1}\right) = \arctan \dfrac{m}{a_1}$, by 
Theorem~\ref{rem:around0} and Remark~\ref{rem:shifted}, we get that the $n-1$ nonzero eigenvalues of $\calJ_{l,m}$ belong 
to~\eqref{eq:NEW22} with $A=\dfrac{m}{a_1}$.

\vspace{2mm}

Conversely, given $\xi\in\mathbb{R}$, $A>0$, and a configuration of points $\{z_j\}_{j=1}^{n-1}$ from~\eqref{eq:NEW22} one applies 
Theorem~\ref{rem:around0} and Remark~\ref{rem:shifted} (with $\alpha:=1+iA$) to get two polynomials $\widehat{p}$ and $q$ of degree $n-1$ 
whose zeros $\{\lambda_j\}_{j=1}^{n-1}$ and $\{\mu_j\}_{j=1}^{n-1}$ satisfy~\eqref{almost.interlacing.22}.

Then $(z-\xi)\widehat{p}(z)$ and $q(z)$ have strictly interlacing zeros. This uniquely determines $\calJ$ such that 
$(z-\xi)\widehat{p}(z)$ and $q(z)$ are the characteristic polynomials of $\calJ$ and $\calJ^{(1)}$. Now having the 
matrix~$\calJ$ one can uniquely determine the numbers $m>0$ and $l\in\mathbb{R}$ by the formul\ae~\eqref{coeff.l.m}.
\end{proof}

%\begin{remark}
%Note that the results of Section~\ref{ss:ranktwo} imply the results of Section~\ref{ss:multipl} once we set $l:=b_1 k$   and $m:=a_1 k$, $k>0$.
%\end{remark}

\begin{remark}
    In a similar manner one can also solve a spectral problem for rank two additive perturbations where both $(2,1)$ and $(1,2)$-entries of~\eqref{eq:jacobiRank2} are equal to $a_1+im$. Indeed, using a unitary conjugation one can reduce such matrices to the form ~\eqref{eq:jacobiRank2}. We leave the details as an exercise to the reader.
\end{remark}

%%%%%%%%%%%%%%%%%%%%%%%%%%%%%%%%%%%%%%%%%%%%%%%%%%%%%%%%%%%%%%%%%
\section{Acknowledgements}
%%%%%%%%%%%%%%%%%%%%%%%%%%%%%%%%%%%%%%%%%%%%%%%%%%%%%%%%%%%%%%%%%
The work of M.\,Tyaglov was partially supported by National Natural Science Foundation of China under grant no.~11871336.

\end{document}